\documentclass[12pt]{amsart}
\usepackage[a4paper, total={6in, 9in}]{geometry}
\usepackage{amsthm}
\usepackage{amsmath}
\usepackage{etoolbox}
\usepackage{amsfonts}
\usepackage[shortlabels]{enumitem}
\usepackage{amssymb}
\usepackage[all,cmtip]{xy}
\newtheorem{thm}{Theorem}[section]

\newtheorem{cor}[thm]{Corollary}

\newtheorem{remark}[thm]{Remark}
\newtheorem{lemma}[thm]{Lemma}
\newtheorem{prop}[thm]{Proposition}

\newtheorem{ex}[thm]{Example}
\newtheorem{defn}[thm]{Definition}

\renewcommand{\span}{\mathrm{span}}

\newcommand{\bb}[1]{\mathbb{#1}}
\newcommand{\cl}[1]{\mathcal{#1}}

\newcounter{egcounter}
\begin{document}

\title[A GKZ-type Theorem for $H^p_{\alpha, \beta}$ spaces]{A Gleason--Kahane--\.Zelazko Theorem for $H^p_{\alpha, \beta}$ spaces}

\author{Jaikishan}
\address{Department Of Mathematics\\
      School of Natural Sciences\\
      Shiv Nadar Institution of Eminence\\
  Gautam Buddha Nagar\\
         Uttar Pradesh 201314, India}
\email{jk301@snu.edu.in}

\author{Sneh Lata}
\address{Department Of Mathematics\\
       School of Natural Sciences\\
      Shiv Nadar Institution of Eminence\\
  Gautam Buddha Nagar\\
         Uttar Pradesh 201314, India}
\email{sneh.lata@snu.edu.in}

\author{ Dinesh Singh}
\address{Centre For Digital Sciences\\
      O. P. Jindal Global University\\
   Sonipat\\
        Haryana 131001, India}
\email{dineshsingh1@gmail.com}

\begin{abstract}
We study two entities that have proved to be of interest and importance in their own right viz. the $H^p_{\alpha, \beta}$ spaces and the classical Gleason–Kahane–Żelazko (GKZ) theorem. We establish a GKZ-type theorem on the $H^p_{\alpha, \beta}$ spaces by identifying a natural class of functions serving as the counterpart of invertible elements, and proving that every continuous linear functional nonvanishing on this class is a point evaluation. As an application, we characterize weighted composition operators on these spaces. It must be noted that the $H^p_{\alpha, \beta}$ spaces do not possess the various structural advantages of the classical Hardy spaces where the GKZ theorem already exists and thus we have had to modify and establish methods that rely on suitable extensions and refinements of the known techniques. Additionally, we provide examples that illustrate the natural class of functions arising in our GKZ-type theorem and demonstrate the sharpness of our characterization of weighted composition operators under the assumption of surjectivity.
\end{abstract}

\maketitle 

\section{Introduction} 
The Gleason–Kahane–\.Zelazko (GKZ) theorem is a classical result in the theory of Banach algebras characterizing multiplicative linear functionals on Banach algebras  through their behavior on invertible elements \cite{Gleason, KZ, zelazko}. We recall its statement for complex unital Banach algebras.

\begin{thm}(GKZ theorem)
A linear functional $\varphi$ on a complex unital Banach algebra $\mathcal{A}$ is multiplicative if and only if $\varphi(e) = 1$ and $\varphi(x) \neq 0$ for every invertible element $x$ in $\mathcal{A}$.
 \end{thm}

In recent years, several analogs of the GKZ theorem have been developed for function spaces that do not necessarily possess an algebraic structure. In 2015, Mashreghi and Ransford \cite{MR2015} established an influential GKZ-type result for Hardy spaces $H^p$ over the open unit disc $\bb D$ for the entire range $0<p\le \infty$. They also obtained an analogous result for a broad class of Banach spaces of analytic functions on $\bb D$, encompassing some well-known examples such as Bergman spaces, weighted Dirichlet spaces, and the disc algebra. Later, together with their collaborators, they proved another GKZ-type theorem for weighted Dirichlet spaces in \cite{MR2018} and subsequently extended these results to the class of reproducing kernel Hilbert spaces with normalized complete Pick kernels in \cite{CHMR}. In each of these works, the corresponding GKZ-type results were further applied to characterize weighted composition operators on the underlying spaces.

In 2017, motivated by a question arising in geophysical imaging, Kou and Liu \cite{kou} characterized bounded linear operators on Hardy spaces $H^p$ for $1<p<\infty$ that preserve outer functions as weighted composition operators. As an intermediate step, they also established a GKZ-type theorem for these Hardy spaces. Building on the work of \cite{kou}, Sampat derived analogous results in 2021 for spaces of analytic functions in several variables \cite{Sam}.  

The nonvanishing hypothesis on invertible elements in the GKZ theorem is pivotal in establishing the multiplicativity of the functional. When working with function spaces, however, the notion of invertibility is no longer directly applicable in the same sense. Consequently, in developing GKZ-type results for function spaces, one is naturally led to identify an appropriate class of functions that serves the role of invertible elements.

In the classical Hardy space setting \cite{MR2015}, this role is played by outer functions, whereas in weighted Dirichlet spaces it is played by nowhere-vanishing functions \cite{MR2018}, and in the setting of reproducing kernel Hilbert spaces with normalized complete Pick kernels by cyclic functions \cite{CHMR}. Note that the Hardy space $H^2$ is itself a reproducing kernel Hilbert space with normalized complete Pick kernel, and in this case, cyclic functions coincide with outer functions.

In \cite{kou}, the authors also worked with outer functions, though not with the entire class; instead, they assumed the linear functional to be non-zero only on the subclass 
$\{e^{zw}: w\in \bb C\}$. We further note that the GKZ-type theorems for Hardy spaces and, more generally, for reproducing kernel Hilbert spaces established in \cite{CHMR, MR2015, MR2018} do not assume continuity of the linear functional; rather, continuity follows automatically from the hypotheses. In contrast, the works \cite{kou, Sam} assume continuity of the linear functional among their hypotheses.

In \cite{JLS}, the authors of the present paper developed a substantially broader framework for GKZ-type theorems by working with vector spaces of functions on a fixed open disc in the complex plane that are equipped with a topology, without assuming any compatibility between the topology and the underlying algebraic structure. In particular, unlike the earlier approaches discussed above, no Banach space structure was required. Furthermore, the functions under consideration were not assumed to be analytic. The nonvanishing hypothesis on the linear functional was also substantially weakened: instead of assuming nonvanishing on the entire class of outer functions, it was required only on the much smaller class consisting of powers of linear outer polynomials. The resulting framework not only recovers the spaces covered by the GKZ-type theorem \cite[Theorem 3.1]{MR2015}, but also applies to several additional well-known spaces. The theorem was then applied to characterize 
weighted composition operators among continuous linear operators whose images of powers of linear outer polynomials are nowhere-vanishing.  

The authors further extended these investigations in \cite{JLS1} to vector-valued functions in several variables, where analogous structural results and corresponding applications to weighted composition operators were obtained. We also note that, similarly to \cite{kou, Sam}, continuity of the linear functional was assumed in these works. Examples were further provided to demonstrate the necessity of the continuity assumption for obtaining the desired conclusions in the corresponding settings.

The aim of the present paper is to develop this circle of ideas in the setting of $H^p_{\alpha,\beta}$ spaces, which arise as closed subspaces of the classical Hardy spaces $H^p$. These subspaces were introduced in \cite{DPRS} for the case $p=2$ in connection with a constrained Nevanlinna--Pick interpolation problem and have since been 
studied from a variety of perspectives. In particular, many classical Hardy space results continue to admit analogs for these subspaces; see, for example, \cite{ABS, BL, Sus}. 
These developments make the $H^p_{\alpha,\beta}$ spaces a natural framework for investigating a GKZ-type result and weighted composition operators.  
 
For $\alpha,\beta\in\bb{C}$ with $|\alpha|^2+|\beta|^2=1$ and $1\le p<\infty$, we define 
$$
H^p_{\alpha,\beta}=\span{\{\alpha+\beta z, z^2H^p\}}. 
$$
Indeed, 
\begin{equation*}
H^p_{\alpha,\beta}
= \span\{\alpha+\beta z\} \oplus z^{2}H^{p},
\end{equation*}
where the direct sum is an orthogonal direct sum for $p=2$; otherwise, an algebraic direct sum. It can be readily verified that $H^p_{\alpha,\beta}$ is a closed subspace of 
$H^p$, since multiplication with $z^{2}$ is an isometry on $H^{p}$. Henceforth, whenever we refer to the space $H^p_{\alpha, \beta},$ it will be understood that $1\le p<\infty$ and that the scalars $\alpha,\beta\in\bb C$ satisfy $|\alpha|^2+|\beta|^2=1$ with $\alpha\neq 0$.

Although $H^p_{\alpha,\beta}$ spaces, being closed subspaces, inherit several features of the Hardy spaces, they also exhibit important structural differences. For instance, the function $1$ belongs to $H^p_{\alpha,\beta}$ only when $\beta=0$, whereas $z\in H^p_{\alpha,\beta}$ only when $\alpha=0$. Throughout this paper, we work under the standing assumption $\alpha\neq 0$ and do not isolate the special case $\beta=0$; consequently, neither of the distinguished functions $1$ and $z$ is available in our setting.  
  
Furthermore, one of the most fundamental structural features of the Hardy spaces is the inner--outer factorization. More precisely, every non-zero function $f\in H^p$ $(p>0)$ admits a unique factorization of the form 
$$
f(z)=I(z)O(z),
$$ 
where $I$ is an inner function and $O$ is an outer function in $H^p$. In particular, this factorization expresses every function as a product of two components: one belonging to the algebra $H^\infty$ of bounded analytic functions on $\bb D$, the multiplier algebra of $H^p,$ and the other being an outer function. 

Since $H^p_{\alpha,\beta}\subseteq H^p$, every function in $H^p_{\alpha,\beta}$ admits a factorization as the product of an $H^\infty$ function and an outer function in the Hardy space. However, it is presently not known whether these spaces possess an intrinsic analog of this factorization. More specifically, since the multiplier algebra of $H^p_{\alpha,\beta}$ is $H^\infty_1$ (Theorem \ref{multiplier}) and $H^\infty_1\subsetneq H^\infty$, it is not known whether every function in $H^p_{\alpha,\beta}$ can be expressed as the product of a function in $H^\infty_1$ and a cyclic function in $H^p_{\alpha,\beta}$. As we shall see in Definition \ref{def-cyclic}, cyclic functions serve as the natural counterpart of outer functions in the Hardy spaces.  
 
The existing arguments in the literature mentioned above for GKZ-type results depend heavily on the simultaneous presence of the functions $1$ and $z$, as 
well as on the availabity of appropriate factorizations. We note that reproducing kernel Hilbert spaces with complete Pick kernels also admit an analogous factorization, in which functions from the multiplier algebra play the role of bounded analytic functions, while cyclic functions replace outer functions. The absence of these features in the spaces $H^p_{\alpha,\beta}$ therefore introduces substantial additional difficulties. Consequently, 
the proofs of our analogous results–Theorems \ref{M}, \ref{weighted1}, and \ref{weighted2}–require significant modifications of the approaches employed in the above-cited works.
Nevertheless, as we show in Section \ref{gkz}, the resulting characterizations closely parallel those obtained in the classical Hardy space setting.  

We conclude the introduction with a brief outline of the paper. Section \ref{Prelim} contains some preliminary and auxiliary results. In Section \ref{mlf}, we provides, under a mild additional assumption, a complete and explicit characterization of the multiplicative linear functionals on $H^\infty_1$, the multiplier algebra of the $H^p_{\alpha, \beta}$ spaces. 
Section \ref{gkz} contains the main results of the paper. Theorem \ref{M} is a GKZ-type result for $H^p_{\alpha, \beta}$ spaces, and Theorems \ref{weighted1} and \ref{weighted2} 
characterize weighted composition operators on these spaces for different codomains. 

\section{Preliminaries and Auxiliary Results}\label{Prelim}
For $1\le p<\infty$, the Hardy space $H^p$ consists of all complex-valued analytic functions on the open unit disc $\bb D$ such that 
$$
\sup_{0\le r<1} \int_{0}^{2\pi} |f(re^{i\theta})|^p\, d\theta <\infty.
$$
Equipped with norm
$$
||f||_p= \left(\sup_{0\le r<1} \int_{0}^{2\pi} |f(re^{i\theta})|^p\, d\theta\right)^{1/p},
$$
$H^p$ is a Banach space.
Further, the Hardy space $H^\infty$ is the Banach algebra of all bounded analytic functions on $\bb D$ endowed with the norm
$$
||f||_\infty = \sup_{z\in \bb D}|f(z)|.
$$

For the reader's convenience, we recall the definition of $H^p_{\alpha,\beta}$ spaces. For $\alpha,\beta\in\bb C$ satisfying $|\alpha|^2+|\beta|^2=1$ and $1\le p<\infty$, the space $H^p_{\alpha,\beta}$ is the closed subspace of $H^p$ given by
$$
H^p_{\alpha,\beta}=\operatorname{span}\{\alpha+\beta z,\; z^2H^p\}.
$$

As stated in the introduction, throughout this paper, whenever the spaces $H^p_{\alpha,\beta}$ are considered, we assume that $1\le p<\infty$ and that the scalars $\alpha,\beta\in\bb C$ satisfy $|\alpha|^2+|\beta|^2=1$ with $\alpha\neq 0$. The following observation provides a necessary and sufficient condition for a function in $H^p$ to belong to $H^p_{\alpha,\beta}$.  

\begin{lemma}\label{Cha}\cite[Lemma 2.1]{ABS}
A function $f\in H^p$ belongs to $ H^{p} _{\alpha,\beta}$ if and only if it satisfies $f(0)\beta=f^\prime(0)\alpha$.
\end{lemma}

One of the primary objectives of this paper is to characterize multiplicative linear functionals on $H^p_{\alpha, \beta}$ spaces. It is well established in the literature \cite{CHMR, MR2015, MR2018} that the multiplier algebras of the Hardy spaces $H^p$, $1\le p<\infty$, and of reproducing kernel Hilbert spaces play a central role in the derivation of the corresponding GKZ-type results. This naturally leads us to consider the multiplier algebras of $H^p_{\alpha,\beta}$ spaces.

Given a vector space $X$ of complex-valued functions on $\bb D$, a complex-valued function $h$ defined on $\bb D$ is called a multiplier of $X$ if $hg$ belongs to $X$ whenever $g$ belongs to $X$. Let $\cl M(X)$ denote the collection of all multipliers of $X$. It is straightforward to verify that $\cl M(X)$ is an algebra under 
pointwise operations; this algebra is called the multiplier algebra of $X$. For the Hardy spaces $H^p$, $0< p\le \infty$, the multiplier algebra is the Banach algebra $H^\infty.$

 In \cite{DPRS}, the authors showed that the multiplier algebra of $H^2_{\alpha, \beta}$ is the closed subalgebra 
$$H^\infty_{1}:=\{f\in H^\infty : f^\prime(0)=0\}
$$ of $H^\infty.$ We show below that this conclusion remains valid for the entire range $1\le p< \infty$. It is straightforward to see that  $H^\infty_1\subseteq \cl M(H^p_{\alpha, \beta})$ for $1\le p<\infty$; establishing the reverse inclusion is the content of the following result.

\begin{thm}\label{multiplier}
For each $1\le p<\infty$, $\mathcal{M}(H^{p} _{\alpha,\beta})=H^\infty_{1}$.
\end{thm}

\begin{proof} As noted above, we already have that $H^\infty_1 \subseteq \cl M(H^p_{\alpha,\beta})$. To prove the reverse inclusion, let $h:\bb{D}\to \bb{C}$ be a multiplier of $H^p_{\alpha,\beta}$. Then $z^2h\in H^p_{\alpha,\beta}$, which means $z^2h(z)=c(\alpha+\beta z)+z^2f(z)$ for some $f\in H^p$ and a scalar $c$. This implies that $c=0$, which further implies that $z^2h(z)=z^2f(z)$; therefore, $h$ is analytic on $\bb{D}\setminus\{0\}$. Furthermore, since $h(\alpha+\beta z)\in H^{p} _{\alpha,\beta}$, there exists a $g\in H^p_{\alpha, \beta}$ such that $h(z)=g(z)/(\alpha+\beta z)$ in some neighborhood of $0$. This yields that $h$ is analytic at $0$; therefore $h$ is analytic in $\bb D$. Let $h(z)=\sum_{n=0}^{\infty}a_nz^n$ for $z\in \bb D$. We must have $h(z)(\alpha+\beta z)=c_1(\alpha+\beta z)+z^2 g_0(z)$ for some $g_0 \in H^p$. By comparing the coefficients, we get
\begin{equation*}
a_0\alpha=c_1\alpha, \qquad a_0\beta+\alpha a_1=c_1\beta,
\end{equation*}
from where it follows that $a_1=0$. Hence, $h^{\prime}(0)=0$.  

Lastly, we shall show that $h$ is bounded. Let $f\in H^p$. Then, $z^2fh\in H^{p} _{\alpha,\beta}$, which means $z^2fh=z^2g_1$ for some $g_1\in H^p$. This, using the identity theorem, implies that $fh=g_1.$ Hence, 
$hH^{p}\subset H^{p}$, which establishes that $h\in H^\infty.$ This completes the proof.  
\end{proof}	

As discussed in the introduction, various analogs of the GKZ theorem for function spaces replace invertible elements by suitable classes of functions, such as outer functions in Hardy spaces, nowhere-vanishing functions in weighted Dirichlet space, and cyclic functions in reproducing kernel Hilbert spaces with normalized complete Pick kernels. Motivated by these developments, and in order to derive an analog of the GKZ theorem for $H^p_{\alpha,\beta}$ space, it is therefore necessary first to identify a suitable class of functions that can play the same structural role as the invertible elements in the classical GKZ setting. Before introducing this class, we recall some definitions. 

A function $f$ in $H^p$ is said to be an outer function if $\span\{fz^n:n\ge 0\}$ is dense in $H^p$, which is equivalent to saying that $fH^\infty$ is dense in $H^p$. If $\cl H$ 
is an RKHS and $\cl M(\cl H)$ is its multiplier algebra, then $h\in \cl H$ is called cyclic if $h\cl M(\cl H)$ is dense in $\cl H$. Since, $\cl M(H^2)=H^\infty$, therefore $h\in H^2$ 
is cyclic if and only it is outer. Taking inspiration from these definitions, we introduce the following notion for $H^p_{\alpha,\beta}$ spaces.

\begin{defn}\label{def-cyclic}
A function $f$ in $H^p_{\alpha,\beta}$ is said to be $cyclic$ if $\overline{fH^{\infty}_1}=H^p_{\alpha,\beta}$.
\end{defn}

Observe that a cyclic function in $H^p_{\alpha,\beta}$ cannot vanish at any point in $\bb{D}$. Indeed, suppose that $f\in H^p_{\alpha, \beta}$ is a cyclic and vanishes 
 at some point $z_0\in \bb D$. Then, by the definition of cyclicity, every function in $H^p_{\alpha, \beta}$ must vanish at $z_0.$ In particular, $z_0^2=0$, which implies that $z_0=0$. Consequently, $\alpha=\alpha+\beta z_0=0,$ contradicting the assumption that $\alpha\ne 0.$  

The results of Section \ref{gkz} show that cyclic functions serve as the appropriate substitute for invertible elements in the setting of $H^p_{\alpha,\beta}$ spaces.
Since these functions play a pivotal in our work, it is important to determine whether such functions exist at all. The following result provides an effective characterization of cyclic functions in these spaces, which will be useful in constructing examples of such functions.

\begin{thm}\label{cyclic}
A function $f\in H^p_{\alpha,\beta}$ is cyclic if and only if it is outer in $H^p$.
\end{thm}

\begin{proof}
Let $f\in H^p_{\alpha,\beta}$ be cyclic. Consider the space $\cl M:=\overline{fH^\infty}$. Since ${fH^\infty_1}\subset{fH^\infty}$, we have 
\begin{equation}
H^p_{\alpha,\beta}\subset \mathcal{M}\subset H^p.
\end{equation} 

As the co-dimension of $H^p_{\alpha, \beta}$ in $H^p$ is one, we conclude that either $\mathcal{M}=H^p_{\alpha,\beta}$ or $\mathcal{M}=H^p$. 
If $\cl M=H^p_{\alpha,\beta}$, then $zf\in H^p_{\alpha,\beta}$, which implies that $zf=z^2g$ for some $g\in H^p$. But this would mean $f(0)=0$, a contradiction to the fact that 
cyclic functions in $H^p_{\alpha, \beta}$ can't vanish at any point. Hence, $\overline{fH^\infty}=\mathcal{M}=H^p$, which establishes that $f$ is outer.  
				 
Now suppose $f$ is outer and $f\in H^p_{\alpha,\beta}$. First we decompose $H^\infty_1$ as 
$$
H^\infty_1=\bb{C}\oplus z^2H^\infty \ \ (\rm algebraic \ direct \ sum).
$$ 
Then,  
\begin{equation*}
\overline{fH^\infty_1}=\overline{f(\bb{C}\oplus z^2H^\infty)}= span\{f\}\oplus \overline{z^2fH^\infty} = {span\{f\}}\oplus{z^2H^p},
\end{equation*}
since $f$ is outer in $H^p$. Thus, $z^2H^p\subseteq \overline{fH^\infty_1}.$ Further, as $f\in H^p_{\alpha, \beta}$, we write $f=c(\alpha+\beta z)+z^2f_1$ for some $f_1\in H^p$ and scalar $c.$ However, $f$ being outer can't vanish at any point in $\bb D$, therefore $c$ is non-zero. This yields that 
$\alpha+\beta z\in {span\{f\}}\oplus{z^2H^p}=\overline{fH^\infty_1}$. Hence, $H^p_{\alpha,\beta}=\overline{fH^\infty_1}$, which establishes that $f$ is cyclic.  
\end{proof}	

We can deploy the above characterization to construct an entire family of cyclic functions in $H^p_{\alpha, \beta}$ spaces. 
\begin{ex}
Let $\alpha, \beta\in \bb C$ such that $\alpha\ne 0$ and $|\alpha|^2+|\beta|^2=1$. Define $s(z)=(1+\frac{z}{2})^{\frac{2\beta}{\alpha}}$. It is clear that $s$ is analytic and nowhere-vanishing on the open disc centred at zero with radius $2$. Thus, $s$ and $1/s$ are both in $H^\infty(\bb D)$; therefore $s$ is an outer function. Further, $\beta s(0)=\alpha s^\prime(0)$ yields that $s\in H^p_{\alpha,\beta}$. Thus, $s$ is cyclic in $H^p_{\alpha,\beta}$ for every $1\le p<\infty$.
\end{ex}  

We conclude this section by observing that nonvanishing on cyclic functions is a necessary hypothesis for a GKZ-type theorem on 
$H^p_{\alpha,\beta}$.

\begin{prop} Let $\Lambda$ be a non-zero continuous  linear functional on $H^p_{\alpha,\beta}$ and $c$ be a non-zero scalar such that $\Lambda(fg)=c\Lambda(f)\Lambda(g) $ whenever $f, g\in H^p_{\alpha, \beta}$ with $fg\in H^p_{\alpha,\beta}$. Then $\Lambda$ is non-zero at every cyclic function in $H^p_{\alpha,\beta}$.
\end{prop}

\begin{proof}
Suppose there is a cyclic function $s\in H^p_{\alpha,\beta}$ such that $\Lambda(s)=0$. Then, $\Lambda(sz^kh)=c\Lambda(s)\Lambda(z^kh)=0$ for every $h\in H^\infty_1$ and 
$k\ge 2$. Let $f\in H^p_{\alpha,\beta}$, and let $\{h_n\}$ be a sequence in $H^\infty_1$ such that $sh_n\longrightarrow f$ in $H^p_{\alpha, \beta}$. Consequently, 
$z^ksh_n\longrightarrow z^kf$ in $H^p_{\alpha, \beta}$ for all $k\ge2$. It follows that $\Lambda(z^ksh_n)\longrightarrow \Lambda(z^kf)$; thus, $\Lambda(z^kf)=0$ 
for every $f\in H^p_{\alpha,\beta}$ and $k\ge 2$. In particular, for $f=z^k$, we obtain $\Lambda(z^k)=0$ for all $k\ge 2.$ Let $s=\mu(\alpha+\beta z)+z^2g$ for some constant $\mu\neq 0$ and $g\in H^p$. This yields that $\Lambda(\alpha+\beta z)=0$; therefore, $\Lambda\equiv0$, which is a contradiction. Hence, $\Lambda$ must non-zero at every cyclic function. 
\end{proof}

\section{Multiplicative linear functionals on $H^\infty_1$}\label{mlf}
The main goal of this section is to investigate multiplicative linear functionals (mlfs) on $H^\infty_1$. As in the case of $H^\infty$, a complete description of these functionals is subtle. Nevertheless, Theorem \ref{H_1} provides a complete and explicit characterization of mlfs under a mild additional assumption. This result will play a key role in the proof of Theorem \ref{M}, where such a condition on mlfs arises naturally.  

As a first step, we describe the mlfs on subalgebras of the disc algebra $A(\bb D)$ of the form $gA(\bb D)$, where $g\in A(\bb D).$ Although the following characterization of mlfs on $gA(\bb D)$ follows from the more general result \cite[Theorem 2.1]{Bjo}, the particular structure of these subalgebras allows for a more direct and elementary proof, which we present here. Recall that every mlf on $A(\bb D)$ is given by point evaluation at some point of the closed unit disc $\overline{\mathbb{D}}$. 
 
\begin{prop}\label{ideal-disc}
Let $g\in A(\bb{D})$ and $\Lambda : gA(\mathbb{D}) \to \mathbb{C}$ be a non-zero multiplicative linear functional. Then there exists $w\in \overline{\mathbb{D}}$ such that 
$\Lambda(h)=h(w)$ for all $h$ in $gA(\bb{D})$. 
\end{prop}

\begin{proof}
Since $\Lambda$ is non-zero, we have $\Lambda(gh_0)\neq0$ for some $h_0\in A(\bb{D})$. This implies that $\Lambda((gh_0)^2)=(\Lambda(gh_0))^2\ne0$. 
However, $\Lambda((gh_0)^2)=\Lambda(g)\Lambda(gh_0^2)$, so $\Lambda(g)\neq0$. This allows to define $\tilde{\Lambda}: A(\bb{D}) \to \bb{C}$ by 
$$
\tilde{\Lambda}(f)=\Lambda(gf)/\Lambda(g)
$$ for all $f\in A(\bb{D})$. It is easy to check that $\tilde{\Lambda}$ is a linear map. Further, 
$$
\tilde{\Lambda}(f_1f_2)=\frac{\Lambda(gf_1f_2)}{\Lambda(g)}=\frac{\Lambda(gf_1f_2)\Lambda(g)}{\Lambda(g)\Lambda(g)}=\frac{\Lambda(gf_1)\Lambda(gf_2)}{\Lambda(g)\Lambda(g)}=\tilde{\Lambda}(f)\tilde{\Lambda}(g), 
$$ 
which proves that $\tilde{\Lambda}$ is multiplicative. Thus, there exists a point $w\in\overline{\mathbb{D}}$ such that $\tilde{\Lambda}(f)=f(w)$; equivalently, 
$\Lambda(gf)=\Lambda(g)f(w)$ for all $f\in A(\mathbb{D})$. Taking $f=g$, we get $\Lambda(g^2)=\Lambda(g)g(w)=\Lambda(g)^2$, which yields that $\Lambda(g)=g(w)$. 
Hence, $\Lambda(gf)=g(w)f(w)=(gf)(w)$ for all $f\in A(\mathbb{D})$. This completes the proof. 
\end{proof} 

The preceding result can be extended to the closure of $g A(\mathbb{D})$. In a classical paper, Rudin proved that every closed ideal of the disc algebra 
$A(\mathbb{D})$ is principal; that is, for any closed ideal $I \subset A(\mathbb{D})$, there exists $g \in A(\mathbb{D})$ such that 
$I = \overline{g\mathcal{A}(\mathbb{D})}.$ Combining this fact with the above lemma, we obtain the following result, which characterizes the maximal ideal space of a closed 
subalgebra of \( \mathcal{A}(\mathbb{D}) \) that is also an ideal.  

 \begin{prop}
 Every non-zero multiplicative linear functional on a closed subalgebra of $A(\bb D)$ that is also an ideal is a point evaluation at a point of $\overline{\bb D}$.
 \end{prop}

We shall use the following consequence of Proposition \ref{ideal-disc} in the proof of our main result (Theorem \ref{H_1}) of this section.
	
\begin{cor}\label{subalg}
Let $n\ge1$ be a positive integer and $\Lambda:\mathcal{A}:=\bb{C}\oplus z^n A(\bb{D})\to \bb{C}$ be a multiplicative linear functional. Then there exists a point $w\in\overline{\mathbb{D}}$ such that $\Lambda(f)=f(w)$ for all $f$ in $\mathcal{A}$. 
\end{cor}
\begin{proof} The restriction of $\Lambda$ to $z^n A(\bb D)$ is still a mlf. Then the result follows using Proposition \ref{ideal-disc}.
\end{proof}

We are now ready to prove the following characterization theorem for a class of mlfs on $H^\infty_1$.

\begin{thm}\label{H_1}
Let $\Lambda:H^\infty_1 \to \bb{C}$ be a multiplicative linear functional satisfying $\Lambda(z^2)\in\bb{D}$. Then there exists a $w\in\bb{D}$ such that 
$$
\Lambda(f)=f(w)
$$
for all $f\in H_1^\infty$.
\end{thm}

\begin{proof} We divide the proof into two cases: namely, $\Lambda(z^2)=0$ and $\Lambda(z^2)\neq 0$. First, suppose that $\Lambda(z^2)=0$. Then, using multiplicativity of $\Lambda$, we obtain $\Lambda(z^4f)=0$ for all $f\in H^\infty$. In particular, $\Lambda(z^6)=0$, which implies that $(\Lambda(z^3))^2=\Lambda(z^6)=0,$ that is, $\Lambda(z^3)=0.$ Let $f\in H^\infty_1$. Since $f^{\prime}(0)=0$, we can decompose it as  
\begin{equation*}
f(z)=a_0+a_2z^2+a_3z^3+z^4f_0
\end{equation*} 
for some $f_0\in H^\infty$.Then, $\Lambda(f)=a_0= f(0)$. Thus, for the case $\Lambda(z^2)=0$, we have $\Lambda(f)=f(0)$ for all $f\in H^\infty_1.$  

We now turn to the case $\Lambda(z^2)\neq 0$. Using Corollary \ref{subalg} for the restriction of $\Lambda$ to $\bb C\oplus z^2A(\bb D)$, we obtain a point $w\in \overline{\bb D}$ such that $\Lambda(f)=f(w)$ for all 
$f\in \bb{C}\oplus z^2 A(\bb{D})$. Then, $w^2=\Lambda(z^2)\in\bb{D}$, which in fact implies that $w\in\bb{D}$ and $w\ne 0$. We shall establish that $\Lambda(f)=f(w)$ for all 
$f\in H^\infty_1$. To this end, let $f\in H^\infty_1$. Define a function $k$ on $\bb D$ as follows 
$$ 
k(z)=
\begin{cases}
\frac{(f(z)-f(w))(f(z)-f(-w))}{z^2-w^2} &\quad \text{if}  \quad z\neq w,-w\\
\frac{f^\prime(w)(f(w)-f(-w))}{2w} &\quad \text{if}  \quad z= w\\
\frac{f^\prime(-w)(f(-w)-f(w))}{-2w} &\quad \text{if}  \quad z= -w\\
\end{cases}.
$$
Then, the function $k$ is bounded analytic on $\bb D$. Moreover, one can readily verify that $k\in H^\infty_1$. Additionally, 
$$
(z^2-w^2)k=(f-f(w))(f-f(-w)),
$$ 
which implies that 
\begin{equation*}
\Lambda(f-f(w))\Lambda(f-f(-w))=\Lambda(k).0=0.
\end{equation*}
Therefore, for each $f\in H^\infty_1$, either $\Lambda(f)=f(w)$ or $\Lambda(f)=f(-w)$. We claim that when $\Lambda(f)=f(-w)$, one in fact has $f(w)=f(-w)$, which in turn would imply $\Lambda(f)=f(w).$ To settle this claim, suppose $f\in H^\infty_1$ satisfies $\Lambda(f)=f(-w)$. Let $f(z)=\sum_{n=0}^{\infty}a_nz^n$. For each $n\ge 1$, we decompose 
$f$ as 
\begin{eqnarray}
f(z) &=&\underbrace{a_0+a_2z^2+a_3z^3+\cdots+a_{2n+1}z^{2n+1}}_{p_{2n+1}(z)}+\underbrace{a_{2n+2}z^{2n+2}+ a_{2n+3}z^{2n+3}\dots}_{f_{2n+1}(z)}\label{S}\\
&=&p_{2n+1}(z)+f_{2n+1}(z)\label{p3}
\end{eqnarray} 
Then, applying $\Lambda$ to Equation (\ref{p3}), we have $f(-w)= \Lambda(p_{2n+1})+\Lambda(f_{2n+1}) = p_{2n+1}(w)+\Lambda(f_{2n+1})$. Since each 
$f_{2n+1}\in H^\infty_1,$ therefore $\Lambda(f_{2n+1})$ is either $f_{2n+1}(w)$ or $f_{2n+1}(-w)$. If there exists some $n\ge 1$ such that 
$\Lambda(f_{2n+1})=f_{2n+1}(w)$, then $f(-w)=f(w).$ Suppose now that $\Lambda (f_{2n+1})=f_{2n+1}(-w)$ for all $n\ge 1$. Then 
$f(-w)=p_{2n+1}(w)+f_{2n+1}(-w)$, which implies that 
\begin{equation}\label{p4}
p_{2n+1}(-w)=p_{2n+1}(w)
\end{equation} 
for every $n\ge 1.$ Then, taking $n=1$ in Equation (\ref{p4}), we obtain  
$$
a_0+a_2w^2-a_3w^3 =a_0+a_2w^2+a_3w^3,
$$
which yields $a_3=0$. Next, assuming $n=5$ in Equation (\ref{p4}), we have    
$$
a_0+a_2w^2+a_4w^4-a_5w^5 =a_0+a_2w^2+a_4w^4+a_5w^5,
$$
which implies $a_5=0$. Continuing in this manner, we conclude that $a_{2n+1}=0$ for all $n\ge 1$, and hence $f$ is an even function. Therefore, $f(w)=f(-w)=\Lambda(f)$. This settles the case $\Lambda(z^2)\ne 0$ and completes the proof. 
\end{proof}

\section {A GKZ-type result and weighted composition operators}\label{gkz}
In this section, we prove our main results, namely Theorems \ref{M}, \ref{weighted1}, and \ref{weighted2}. Theorem \ref{M} establishes a GKZ-type theorem for $H^p_{\alpha,\beta}$ spaces, while Theorems \ref{weighted1} and \ref{weighted2} characterize weighted composition operators on these spaces for different codomains.

Although the proof of Theorem \ref{M} is motivated by the arguments used to obtain the GKZ-type theorem for Hardy spaces in \cite{MR2015}, it requires substantially different techniques, since several structural features of the classical Hardy spaces are unavailable in the setting of $H^p_{\alpha,\beta}$.  

Theorems \ref{weighted1} and \ref{weighted2} are obtained as consequences of Theorem \ref{M}. In Hardy spaces, as well as in the other settings discussed previously, analogous characterizations follow from corresponding GKZ-type theorems through arguments that rely crucially on the simultaneous availability of the constant function $1$ and the identity function $z$. However, our setting lacks this feature; consequently, the existing methods do not apply directly. As a result, the proofs of Theorems \ref{weighted1} and \ref{weighted2} require significant modifications of the approach used in the existing literature. Nevertheless, the resulting characterizations remain similar in form.
	
\begin{thm}\label{M} Let $\Lambda:H^p_{\alpha,\beta} \to \bb{C}$ be a continuous linear functional such that $\Lambda(f)\neq 0$ for all cyclic $f\in H^p_{\alpha,\beta}$. Then there is a non-zero scalar $c$ and a point $w\in\bb{D}$ such that $\Lambda(f)=cf(w)$ for all $f\in H^p_{\alpha,\beta}$. 
\end{thm}

\begin{proof}
Fix a cyclic function $s\in H^p_{\alpha,\beta}$, and define $\Lambda_s:H^\infty_{1} \to \bb{C}$ by 
$$
\Lambda_s(f)=\Lambda(sf)/\Lambda(s).
$$ 
It is easy to see that $\Lambda_s$ is linear and $\Lambda_s(1)=1$. Further, for an invertible $g$ in $H^\infty_{1}$, the function $sg$ is cyclic; hence $\Lambda$ is non-zero on $sg$. Thus, $\Lambda_s$ is non-zero on invertible elements in $H^\infty_{1}$. Consequently, by the $GKZ$ theorem, $\Lambda_s$ is multiplicative on $H^\infty_{1}$. 
	 	
Moreover, for each scalar $\lambda$ with $|\lambda|\ge 1$, the polynomial $z^2-\lambda$ is outer and belongs to $H^\infty_1$. Therefore, $s(z^2-\lambda)$ is outer and 
belongs to $H^p_{\alpha, \beta}$, and hence cyclic whenever $|\lambda|\ge 1$. This yields that $\Lambda_s(z^2-\lambda)=\Lambda(s(z^2-\lambda))\ne 0$ for every 
$\lambda$ outside $\bb D$; thus, $\Lambda_s(z^2)\in \bb{D}$. Therefore, using Theorem \ref{H_1}, we obtain a $w\in\bb{D}$ such that 
\begin{equation*}
\Lambda_s(f)=f(w) 
\end{equation*}
for all $f\in H^\infty_{1}$. It follows that $\Lambda(sf) = \Lambda_s(f)\Lambda(s) = f(w)\Lambda(s)$ for all $f\in H^\infty_{1}$.  

Let $h$ be any function in $H^p_{\alpha, \beta}$. Then there exists a sequence $\{g_n\}$ in $H^\infty_{1}$ such that $sg_n\longrightarrow h$ in $H^p_{\alpha,\beta}$. Continuity of $\Lambda$ implies that $\Lambda(sg_n)\longrightarrow \Lambda(h)$. We also have $\Lambda(sg_n)=\Lambda(s)g_n(w)$ and $\Lambda(s)\neq0$; therefore,  it follows that   
\begin{equation}\label{lambda}
g_n(w)\longrightarrow \Lambda(h)/\Lambda(s).
\end{equation}

On the other hand convergence in $p$ norm implies pointwise convergence, which means, $s(w)g_n(w)\longrightarrow h(w)$. Thus, we conclude that 
 \begin{equation}\label{c}
\Lambda(h)=\frac{\Lambda(s)}{s(w)}h(w)
\end{equation}
for all $h\in H^p_{\alpha,\beta}$. 

We claim that the constant $\Lambda(s)/s(w)$ is independent of the choice of the cyclic function $s$. Once this claim is established, (\ref{c}) yields the desired form for 
$\Lambda.$ For this purpose, we take another cyclic function $s_0$ and define 
$\Lambda_{s_{0}}:H^\infty_{1} \to \bb{C}$ by $\Lambda_{s_{0}}(f)=\Lambda(s_{0}f)/\Lambda(s_{s_{0}})$. Then, using the similar arguments as above, we obtain a 
$\zeta\in\bb{D}$ satisfying 
\begin{equation}\label{r}
\Lambda(h)=\frac{\Lambda(s_{0})}{s_0(\zeta)}h(\zeta)
\end{equation} 
for all $h\in H^p_{\alpha,\beta}$. From (\ref{c}) and (\ref{r}), we deduce that $\frac{\Lambda(s)}{s(w)}h(w)=\frac{\Lambda(s_{0})}{s_0(\zeta)}h(\zeta)$ for all $h\in H^p_{\alpha,\beta}$. In particular, for $h(z)=z^n$ with $n\ge2$, we have 
$$
\frac{\Lambda(s)}{s(w)}w^n=\frac{\Lambda(s_{0})}{s_0(\zeta)}{\zeta}^n
$$
If either $w$ or $\zeta$ is zero, then the other must also be zero. Suppose $w$ and $\zeta$ are both non-zero. Then, we have 
$$
(w/\zeta)^2=(w/\zeta)^3,
$$
which readily implies that $w=\zeta$. Thus, $\frac{\Lambda(s)}{s(w)}=\frac{\Lambda(s_{0})}{s_0(\zeta)}$; hence, (\ref{c}) gives the desired form of $\Lambda$. This 
completes the proof.
\end{proof}
	 	 
We now use Theorem \ref{M} to derive the following charaterization of weighted composition operators on $H^p_{\alpha, \beta}$ spaces into the space of analytic functions on $\bb D$, equipped with the Fr\'echet-space topology. 

\begin{thm}\label{weighted1}
Let $T:H^p_{\alpha,\beta} \to {\rm Hol}(\bb{D})$ be a continuous linear functional such that $(Tg)(z)\neq0$ for every cyclic function $g\in H^p_{\alpha,\beta}$ and $z\in\bb{D}$, where ${\rm Hol}(\bb{D})$ is considered with the Fr\'echet-space topology. Then there exist holomorphic functions $\phi:\bb{D}\to\bb{D}$ and $\psi:\bb{D}\to \bb{C}\setminus \{0\}$ such that 
\begin{equation*}
Tf=\psi.(f\circ{\phi})
\end{equation*} 
for every $f\in H^p_{\alpha,\beta}$. 
\end{thm}
	 	 
\begin{proof} Fix a cyclic function $s\in H^p_{\alpha,\beta}$. Then, for a given point $\zeta\in \bb{D}$, the map $T_\zeta: H^p_{\alpha, \beta} \to \bb{C}$ defined by 
$$
T_\zeta(f)=(Tf)(\zeta)/(Ts)(\zeta)
$$ is linear, continuous, and non-zero on every cyclic function in $H^p_{\alpha,\beta}$. Thus,  using Theorem \ref{M}, we obtain a point $w_\zeta\in\bb{D}$ such that $(Tf)(\zeta)/(Ts)(\zeta)= f(w_\zeta)/s(w_\zeta)$ for every $f\in H^p_{\alpha, \beta}$. Then, 
\begin{equation*}
\frac{(Tz^3)(\zeta)}{(Ts)(\zeta)}= \frac{w^3_\zeta}{s(w_\zeta)} \quad \text{and} \quad
\frac{(Tz^2)(\zeta)}{(Ts)(\zeta)}= \frac{w^2_\zeta}{s(w_\zeta)}.
\end{equation*}

Next, letting $f_2:=T(z^2)$ and $f_3:=T(z^3)$, we obtain   
\begin{equation}\label{f2/f3}
f_3(\zeta)= \frac{(Ts)(\zeta)}{s(w_\zeta)}w_\zeta^3 \quad \text{and} \quad
f_2(\zeta)=\frac{(Ts)(\zeta)}{s(w_\zeta)}w_\zeta^2.
\end{equation}
We now define a function $\phi: \bb{D}\to \bb{D}$ by 
$$
\phi(\zeta)=w_\zeta.
$$ We claim that $\phi$ is analytic on $\bb D$. Clearly, $\phi=\frac{f_3}{f_2}$ on 
$\bb D\setminus Z_{f_2}$, where $Z_{f_2}:=\{\zeta\in\bb{D} : f_2(\zeta)=(Tz^2)(\zeta)=0\}$. Therefore, $\phi$ is analytic on $\bb D\setminus Z_{f_2}$. Note that if $Z_{f_2}=\bb D$, then $f_2(\zeta)=0$ for every $\zeta\in \bb D$, which, using (\ref{f2/f3}), yields that 
$w_\zeta=0$ for every $\zeta\in \bb D$, which in turn implies that $\phi\equiv 0$; therefore, trivially analytic on $\bb D$. We now turn to case $\bb D\ne Z_{f_2}$. As noted above, 
we already have that $\phi$ is analytic on $\bb D\setminus Z_{f_2}$. Thus, to establish the analyticity of $\phi$ on $\bb D$, it is enough to show that it is continuous on $Z_{f_2}$. 
To this end, let $\zeta_0 \in Z_{f_2}$ and $\{\zeta_n\}$ be a sequence in $\bb D$ such that $\zeta_n \to \zeta_0$. Since we are assuming that $\bb D\ne Z_{f_2}$, we have  $f_2\not\equiv 0$, and therefore there exists some $0<r<1$ such that $f_2$ never vanishes on $D(\zeta_0,r)$, the open disc centred at $\zeta_0$ and radius $r$. This allows us to assume that $f_2(\zeta_n)\ne 0$ for every $n$. Firstly, $f_2(\zeta_0) =0$ implies $w_{\zeta_0}=0$; thus, $\phi(\zeta_0)=0$. Further, using $f_2(\zeta_n)\to 0$ and $(Ts)(\zeta_n)\to (Ts)(\zeta_0)$ in (\ref{f2/f3}), we conclude that $\frac{w_{\zeta_n}^2}{s(w_{\zeta_n})}\to 0$. However, $s$ being analytic on $\bb D$ is bounded on $\overline{D(\zeta_0,r)}$. It follows that $w_{\zeta_n}\to 0$; consequently, $\phi(\zeta_n)=w_{\zeta_n}\to 0=\phi(\zeta_0)$. Thus, $\phi$ is continuous on $Z_{f_2}.$ Hence, $\phi$ is analytic on $\bb D.$ 
 
 Lastly, define 
 $$
 \psi(\zeta)=(Ts)(\zeta)/s(\phi(\zeta)) \ \ {\rm for} \  \zeta\in \bb D.
 $$ 
 Then, $\psi$ maps $\bb D$ into $\bb C\setminus \{0\}$, it is analytic function, and 
$$
Tf=\psi .(f\circ\phi)
$$
for all $f\in H^2_{\alpha,\beta}$.
\end{proof}

The following result further sharpens the characterization in Theorem \ref{weighted1} under the stronger assumption that the operator maps the underlying space $H^p_{\alpha,\beta}$ onto itself. 

Before proving this result, we note that our argument showing that the associated symbol is necessarily an automorphism is partially motivated by \cite[Theorem 2.2]{Bou}. However, similar to the other arguments discussed at the beginning of this section, the proof in \cite{Bou} relies crucially on the simultaneous availability of a non-zero constant function and a translate of the function $z$, a structural property that fails in the present setting. Consequently, additional arguments are developed to overcome these structural obstacles and establish the desired conclusion.

\begin{thm}\label{weighted2}
Let $T:H^p_{\alpha,\beta} \to H^p_{\alpha,\beta} $ be a continuous surjective linear map such that $(Tg)(z)\neq0$ for every cyclic function $g\in H^p_{\alpha,\beta}$ 
and $z\in\bb{D}$. Then $T$ is invertible, and there exist a unimodular scalar $\lambda$ and an invertible function $\psi$ in $H^{\infty}$ such that 
$$
(Tf)(z)=\psi(z) f(\lambda z)
$$
for every $f\in H^p_{\alpha,\beta}$ and $z\in \mathbb{D}$.
\end{thm}

\begin{proof} Using Theorem \ref{weighted1}, there exist holomorphic functions $\phi: \bb{D} \to\bb{D}$ and $\psi: \bb{D} \to \bb{C}\setminus \{0\}$ such that 
$$
Tf=\psi(f\circ\phi).
$$

If $Tf=0$, then $f\circ\phi=0$. Note that surjectivity of $T$ ensures that $\phi$ is not a constant function. Therefore, $f\circ \phi=0$ implies that $f=0$, using the Identity Theorem. So, $T$ is injective; hence, it is invertible.  
	 	 	
We now show that $\phi$ is an automorphism. To this end, it suffices to prove that $\phi$ is a univalent inner function, since every univalent inner function is an automorphism 
of $\bb D$.  
            
 As $T$ is surjective, we have functions $f_2$ and $f_3$ in $H^p_{\alpha,\beta}$ such that $Tf_2=z^2$ and $Tf_3=z^3$. It follows that 
 $\psi (f_2\circ\phi)=z^2$, $\psi (f_3\circ\phi)=z^3$, and  
\begin{equation}\label{mero}
\frac{f_3}{f_2 }(\phi(z))=z
\end{equation} 
for each non-zero $z\in \bb{D}$. This yields that $\phi$ is injective on $\bb{D}\setminus \{0\}$. Now suppose there is $0\neq a\in \bb{D}$ with $\phi(a)=\phi(0)$. Then 
$f_2(\phi(a))=f_2(\phi(0))$. Also, $\psi(a)f_2(\phi(a))=a^2$, $\psi(0)f_2(\phi(0))=0$, and $\psi(0)\ne 0.$ This shows that $a=0;$ therefore $\phi$ is a univalent map in $\bb{D}$.
	 	  
We next show that $\phi$ is inner. We prove it by contradiction. Assuming $\phi$ is not inner, we shall prove the existence of a point on the unit circle $\bb T$ and an open disc $D$ containing it with the property that every function in $H^p_{\alpha, \beta}$ admits an analytic extension to $\bb D\cup D$. However, for any given point on the unit circle, there exists a function in $H^p$, and hence in $H^p_{\alpha, \beta}$, that does not admit an analytic extension at that point. This yields a contradiction; hence, $\phi$ must be inner.  

Suppose $\phi$ is not inner. Let $\tilde{\phi}$ denote the boundary function of $\phi.$ Then there exists a subset of $\bb T$ of positive measure on which $|\tilde{\phi}(e^{i\theta})|<1$. 
This allows us to choose a $0<\delta<1$ such that the measure of the set $S:=\{e^{i\theta}\in \bb{T}: |\tilde{\phi}(e^{i\theta})|<\delta\}$ is also positive. Since $\phi$ is 
non-constant, $\tilde{\phi}(S)$ must be an infinite set. Moreover, this set is contained in a compact subset of 
$\bb D$. Therefore, $\tilde{\phi}(S)$ has a limit point in $\bb D$.  

Let $g$ be a cyclic function in $H^p_{\alpha,\beta}$. By passing to a subset of $S$ of full measure, we may assume that for each $e^{i\theta}\in S$, $\widetilde{Tg}(e^{i\theta})\in\mathbb C$ and the radial limits
\begin{equation}\label{mero15}
\tilde{\phi}(e^{i\theta}) = \lim_{r\to1^-}\phi(re^{i\theta}) \ \ \ \text{and} \ \ \  \widetilde{Tg}(e^{i\theta}) = \lim_{r\to1^-}Tg(re^{i\theta})
\end{equation}
exist, where $\widetilde{Tg}$ denotes the boundary function of $Tg$.

Let $Z$ be the zero set of $f_2$. Then $\mathbb D\setminus Z$ is open and connected. We claim that
$$
\tilde{\phi}(S)\subseteq \mathbb D\setminus Z.
$$

To prove this, let $e^{i\theta}\in S$. Then, 
$$
Tg(re^{i\theta}) = \psi(re^{i\theta})\,g(\phi(re^{i\theta})). 
$$

Thus, using (\ref{mero15}) together with the fact that $g\bigl(\tilde{\phi}(e^{i\theta})\bigr)\neq 0$ (because a cyclic function has no zeroes in $\mathbb D$), it follows that  
$ \lim\limits_{r\to1^-}\psi(re^{i\theta})$ exists and is finite. Next, consider 
$$
\psi(re^{i\theta})\,f_2(\phi(re^{i\theta})) = r^2e^{2i\theta}.
$$
Passing to the limit as $r\to1^-$, we obtain
$$
\left(\lim_{r\to1^-}\psi(re^{i\theta})\right) f_2\bigl(\tilde{\phi}(e^{i\theta})\bigr) = e^{2i\theta},
$$
which implies that $f_2\bigl(\tilde{\phi}(e^{i\theta})\bigr)\neq 0,$ that is, $\tilde{\phi}(e^{i\theta})\in \bb D\setminus Z$. Hence $\tilde{\phi}(S)\subseteq \bb D\setminus Z.$

We now consider the function   
$$
q:=\frac{f_3}{f_2 }
$$ 
that is meromorphic in $\bb D$. Let $W$ be the set of poles of $q$. Then $\bb D\setminus W$ is an open connected set on which $q$ is analytic. Moreover, since  $W\subseteq Z$ and $\tilde{\phi}(S)\subseteq \bb D\setminus Z$, we have $\tilde{\phi}(S)\subseteq \bb D\setminus W.$ In adddition, using (\ref{mero}), 
$$
q(\phi(re^{i\theta}))=re^{i\theta}
$$ 
for every 
$0< r<1$; consequently,  
\begin{equation}\label{mero2}
{q(\tilde{\phi}(e^{i\theta}))}=e^{i\theta}
\end{equation}
for each $e^{i\theta}\in S$. Accordingly, $\overline{\tilde{\phi}(S)}\subseteq \bb D\setminus W$. Therefore, the limit point of $\tilde{\phi}(S)$ in $\bb D$ belongs to $\bb D\setminus W.$ This, together with the facts that $q$ is non-constant on $\bb D\setminus W$ and $q^{\prime}$ is analytic on $\bb D\setminus W$, implies that $q^{\prime}$ must be non-zero at some point of $\tilde{\phi}(S)$. Suppose $e^{i\theta_0}\in S$ such 
that $q^{\prime}(\tilde{\phi}(e^{i\theta_0}))\ne 0$. Then there exists an open disc $D_0$ centered at $\tilde{\phi}(e^{i\theta_0})$ contained in $\bb D\setminus W$ such that 
$q: D_0\to q(D_0)$ admits a holomorphic inverse. Since $f_2(\tilde{\phi}(e^{i\theta_0}))\ne 0$, shrinking $D_0$ if necessary, we may assume that $f_2$ has no zero in $D_0$, 
while $q: D_0\to q(D_0)$ continues to admit a holomorphic inverse. Let us denote this inverse by $h: q(D_0)\to D_0.$ Then,
\begin{equation}\label{mero3}
h(q(z))=z
\end{equation}
for each $z\in D_0.$ Furthermore, using (\ref{mero2}), $e^{i\theta_0} = q(\tilde{\phi}(e^{i\theta_0}))\in q(D_0)$, which, due to analyticity of $q$, is an open set. Therefore, $q(D_0)$ contains an open disc $D_1$ centered at $e^{i\theta_0}$. This allows us to choose a positive number $r_0$ such that for every $r\ge r_0$, we have  $re^{i\theta_0}\in D_1$ and 
$\phi(re^{i\theta_0})\in D_0$. Now, using (\ref{mero}) and (\ref{mero3}), we deduce that 
$$
h(re^{i\theta_0})=h(q(\phi(re^{i\theta_0})))=\phi(re^{i\theta_0})
$$
for all $r\ge r_0$; therefore, $h=\phi$ on the entire $\bb D\cap D_1.$ Additionally, $h(D_1)\subseteq D_0\subseteq \bb D.$ Hence, the function $\phi_1:\bb D\cup D_1\to \bb D$ 
given by 
$$
\phi_1(z)=
\begin{cases}
\phi (z) & \ \ {\rm for} \ \ \ z\in\bb D\\
h(z) & \ \ {\rm for} \ \ \ z\in D_1
\end{cases}
$$
defines an analytic extension of $\phi$ from $\bb{D}$ to $\bb{D}\cup D_1$.  
	  
Further, $h(D_1)\subseteq D_0\subseteq \bb D\setminus Z$ yields that $\frac{z^2}{f_2\circ h}$ is analytic in $D_1$. Moreover, 
$f_2\circ h=f_2\circ \phi$ on $\bb D\cap D_1$ and $\psi (f_2\circ \phi)=z^2$. Consequently, the function 
$\psi_1:\bb D\cup D_1\to \bb D$ defined by
$$
\psi_1(z) = 
\begin{cases}
\frac{z^2}{f_2(h(z))} & {\rm for} \ \ \  z\in D_1\\
\psi(z) & {\rm for} \ \ \ z\in\bb D  
\end{cases}
$$
provides an analytic extension of $\psi$ to $\bb D\cup D_1$. Hence, $\phi$ and $\psi$ each have analytic extensions, namely $\phi_1$ and $\psi_1$, to $\bb D\cup D_1.$ Together with the surjectivity of $T$, this implies that every function $f\in H^p_{\alpha,\beta}$ admits an analytic extension to $\bb{D}\cup D_1$. However, this is not the case. This contradiction proves that $\phi$ must be inner.

In conclusion, we have shown that $\phi$ is a univalent inner function; therefore, it must be an automorphism of the disc $\bb D.$ Let 
$\phi(z)=\lambda\frac{z-a}{1-\overline{a}z}$ for some $a\in\mathbb{D}$ and $\lambda\in \bb T$. We shall prove that $a=0.$ To this end, note that  
$\psi(f\circ\phi)\in H^p_{\alpha,\beta}$ for every $f\in H^p_{\alpha,\beta}$. It follows that $\psi\phi^n\in H^p_{\alpha, \beta}$ for all $n\ge 2$. Then, using Lemma \ref{Cha}, we have  
$$
\left(\beta\psi(0) - \alpha \psi^{\prime}(0)\right)\phi^{n}(0)=n\alpha\psi(0)\phi'(0)\phi^{n-1}(0)
$$ for all $n\ge 2$. This forces $\phi(0)=0$, since $\psi(0)\ne 0$ and $\phi^{'}(0)=\lambda(1-|a|^2)\ne 0$. Therefore, $a=0$, and hence $\phi(z)=\lambda z$.

Lastly, it remains to prove that $\psi$ is invertible in $H^\infty$. To begin with, we show that $\psi\in H^\infty.$ Let $f\in H^p$, and define $g(z)=z^2f(\bar{\lambda}z)$. Then $g\in H^p_{\alpha,\beta}$, and hence $Tg\in H^p_{\alpha,\beta}$. A direct computation shows that 
$$
Tg=\lambda^2 z^2\psi f.
$$ 
Thus, $\psi f\in H^p$. It follows that $\psi H^p\subset H^p$; therefore, $\psi \in H^\infty$.  

To prove that $\psi$ is invertible in $H^\infty$, we use the surjectivity of $T.$ Arguing as above, one obtains that for every $f\in H^p$, there exists $h\in H^p$ such that 
$f=\psi h.$ This shows that $\frac{f}{\psi}\in H^p$, and hence $\frac{1}{\psi}H^p\subseteq H^p$. Consequently, $\frac{1}{\psi}\in H^\infty$. Therefore, $\psi$ is invertible in 
$H^\infty$, which completes the proof.
\end{proof}   

\begin{remark} As noted earlier, Theorem \ref{weighted2} is an analog of \cite[Theorems 2.3 and 3.3]{MR2015}. We highlight two aspects in which its conclusion for the spaces $H^p_{\alpha,\beta}$ differs from the corresponding results in \cite{MR2015}. The results in \cite{MR2015} show that such preservers are weighted composition operators of the form 
$$
f\mapsto \psi (f\circ \phi)
$$ 
where $\phi$ is an automorphism of $\bb D$ and $\psi$ is an invertible element of the multiplier algebra of the underlying space. In contrast, Theorem \ref{weighted2} shows that for $H^p_{\alpha,\beta}$ spaces, the automorphism $\phi$ is necessarily a rotation, rather than an arbitrary automorphism of $\bb D$. Furthermore, the theorem asserts that $\psi$ is an invertible function in $H^\infty$, and $\psi$ need not belong to $H^\infty_1$, the multiplier algebra of $H^p_{\alpha,\beta}$. We illustrate this latter fact with the following example.
\end{remark}

\begin{ex}
Take $\phi(z)=-z$ and $\psi(z)=\sqrt{2}+z$. Then, for $\alpha = 2\sqrt{2}/3$ and $\beta=1/3$, it is a straightforward to verify that the map $T(f)=\psi(z)f(-z)$ is a well-defined one-to-one continuous linear operator on $H^2_{\alpha,\beta}$.  

Note that any $g\in H^2_{\alpha, \beta}$ can be decomposed as 
$$
g(z)=\frac{c}{3}(2\sqrt{2}+z)+z^2g_1(z)
$$
for some scalar $c$ and $g_1\in H^2$.  Further, observe that 
$$
\frac{1}{\sqrt{2}-z}=\frac{\sqrt{2}+z}{2}+z^2h
$$
for some $h\in H^\infty.$ Then, it follows that 
$$
\frac{g(-z)}{\sqrt{2}-z}= \frac{\sqrt{2}c}{6}(2\sqrt{2}+z)+z^2g_2
$$ 
for some $g_2\in H^2$. Therefore, $f(z)=\frac{g(-z)}{\sqrt{2}-z}$ belongs to $H^2_{\alpha, \beta}$. Also, $T(f)=g$, which establishes that $T$ is surjective. Hence, 
$T$ is invertible; thus, satisfy the hypothesis of Theorem \ref{weighted2}. However, note that $\psi$ is invertible in $H^\infty$ but it does not belong to $H^\infty_1.$
\end{ex}

\subsection*{Acknowledgements} The authors thank the Mathematical Sciences Foundation, Delhi, for the support and facilities needed to complete the present work. The first named author thanks the University Grants Commission(UGC), India, for the support, and the second named author thanks Shiv Nadar Institution of Eminence for partially supporting this research.

\end{document}